\documentclass[11pt]{article}
\usepackage{amssymb,amsmath,graphicx,amsfonts}
\usepackage{amsmath}
\usepackage{amsfonts}

\newtheorem{precor}{{\bf Corollary}}

\newenvironment{cor}{\begin{precor}{\hspace{-0.5
               em}{\bf.\ }}}{\end{precor}}
\newtheorem{precon}{{\bf Conjecture}}

\newtheorem{predefin}{{\bf Definition}}

\newenvironment{defin}[1]{\begin{predefin}{\hspace{-0.5
                   em}{\bf.\ }}{\rm
#1}\hfill{$\spadesuit$}}{\end{predefin}}
\newtheorem{preexm}{{\bf Example}}

\newtheorem{preappl}{{\bf Application}}

\newtheorem{prelem}{{\bf Lemma}}

\newenvironment{lem}{\begin{prelem}{\hspace{-0.5
               em}{\bf.\ }}}{\end{prelem}}
\newtheorem{preproof}{{\bf Proof.\ }}

\newenvironment{proof}[1]{\begin{preproof}{\rm
               #1}\hfill{$\blacksquare$}}{\end{preproof}}
\newtheorem{presproof}{{\bf Sketch of Proof.\ }}

\newtheorem{prethm}{{\bf Theorem}}

\newenvironment{thm}{\begin{prethm}{\hspace{-0.5
               em}{\bf.\ }}}{\end{prethm}}
\newtheorem{preconj}{{\bf Conjecture}}

\newtheorem{prealphthm}{{\bf Theorem}}

\newenvironment{alphthm}{\begin{prealphthm}{\hspace{-0.5
               em}{\bf.\ }}}{\end{prealphthm}}
\newtheorem{prepro}{{\bf Proposition}}

\newtheorem{preprb}{{\bf Problem}}

\newtheorem{prerem}{{\bf Remark}}

\newenvironment{rem}{\begin{prerem}{\hspace{-0.5
               em}{\bf.\ }}}{\end{prerem}}
\def\conct[#1,#2]{\mbox {${#1} \leftrightarrow {#2}$}}
\def\dconct[#1,#2]{\mbox {${#1} \rightarrow {#2}$}}
\def\deg[#1,#2]{\mbox {$d_{_{#1}}(#2)$}}
\def\mindeg[#1]{\mbox {$\delta_{_{#1}}$}}
\def\maxdeg[#1]{\mbox {$\Delta_{_{#1}}$}}
\def\outdeg[#1,#2]{\mbox {$d_{_{#1}}^{^+}(#2)$}}
\def\minoutdeg[#1]{\mbox {$\delta_{_{#1}}^{^+}$}}
\def\maxoutdeg[#1]{\mbox {$\Delta_{_{#1}}^{^+}$}}
\def\indeg[#1,#2]{\mbox {$d_{_{#1}}^{^-}(#2)$}}
\def\minindeg[#1]{\mbox {$\delta_{_{#1}}^{^-}$}}
\def\maxindeg[#1]{\mbox {$\Delta_{_{#1}}^{^-}$}}
\def\isdef{\mbox {$\ \stackrel{\rm def}{=} \ $}}
\def\dre[#1,#2,#3]{\mbox {${\cal E}_{_{#3}}(#1,#2)$}}
\def\pdre[#1,#2,#3]{\mbox {${\cal P}_{_{#3}}(#1,#2)$}}
\def\var[#1,#2]{\mbox {${\rm Var}_{_{#1}}(#2)$}}
\def\ls[#1]{\mbox {$\xi^{^{#1}}$}}
\def\hom[#1,#2]{\mbox {${\rm Hom}({#1},{#2})$}}
\def\onvhom[#1,#2]{\mbox {${\rm Hom^{v}}(#1,#2)$}}
\def\onehom[#1,#2]{\mbox {${\rm Hom^{e}}(#1,#2)$}}
\def\core[#1]{\mbox {$#1^{^{\bullet}}$}}
\def\cay[#1,#2]{\mbox {${\rm Cay}({#1},{#2})$}}
\def\cays[#1,#2]{\mbox {${\rm Cay_{s}}({#1},{#2})$}}
\def\dirc[#1]{\mbox {$\stackrel{\rightarrow}{C}_{_{#1}}$}}
\def\cycl[#1]{\mbox {${\bf Z}_{_{#1}}$}}

\def\sdg[#1]{\mbox {$\stackrel{\leftrightarrow}{#1}$}}

\begin{document}
\begin{center}
{\Large \bf On Colorings of graph fractional powers}\\
\vspace*{0.5cm}
{\bf Moharram N. Iradmusa}\\
{\it Department of Mathematical Sciences}\\
{\it Shahid Beheshti University, G.C.}\\
{\it P.O. Box {\rm 19839-63113}, Tehran, Iran}\\
{\it m\_iradmusa@sbu.ac.ir}\\
\vspace*{0.5cm}
\emph{Dedicated to Professor Cheryl Praeger}\\
\emph{on her 60th birthday}\\
\end{center}
\begin{abstract}
\noindent For any $k\in \mathbb{N}$, the $k-$subdivision of graph $G$ is a simple
graph $G^{\frac{1}{k}}$, which is constructed by replacing each edge
of $G$ with a path of length $k$. In this paper we introduce the $m$th power
of the $n-$subdivision of $G$, as a fractional power of $G$, denoted by $G^{\frac{m}{n}}$. In this regard, we investigate chromatic number and clique number of fractional power of graphs. Also, we conjecture that $\chi(G^{\frac{m}{n}})=\omega(G^\frac{m}{n})$ provided that $G$ is a connected graph with $\Delta(G)\geq3$ and $\frac{m}{n}<1$. It is also shown that this conjecture is true in some special cases.

\begin{itemize}
\item[]{{\footnotesize {\bf Key words:}\ Chromatic number, Subdivision of a graph, Power of a graph.}}
\item[]{ {\footnotesize {\bf Subject classification: 05C} .}}
\end{itemize}
\end{abstract}
\section{Introduction and Preliminaries}
In this paper we only consider simple graphs which are finite,
undirected, with no loops or multiple edges. We mention some of the
definitions which are referred to throughout the paper. As usual, we
denote by $\Delta(G)$ or simply by $\Delta$, the maximum degrees of
the vertices of graph $G$, and by $\omega(G)$, the maximum size of a
clique of $G$, where a clique of $G$ is a set of mutually adjacent
vertices. In addition $N_G(v)$, called the neighborhood of $v$ in
$G$, denotes the set of vertices of $G$ which are adjacent to the
vertex $v$ of $G$. For another necessary definitions and notations
we refer the reader to textbook [4].\\
Let $G$ be a graph and $k$ be a positive integer. The
\emph{$k-$power of $G$}, denoted by $G^k$, is defined on the vertex
set $V(G)$, by connecting any two distinct vertices $x$ and $y$ with
distance at most $k$ [1,8]. In other words, $E(G^k)=\{xy:1\leq
d_G(x,y)\leq k\}$. Also \emph{$k-$subdivision of $G$}, denoted by
$G^\frac{1}{k}$, is constructed by replacing each edge $ij$ of $G$ with a
path of length $k$, say $P_{ij}$. These $k-$pathes are called \emph{hyperedges}
and any new vertex is called \emph{internal vertex} or brifely \emph{i-vertex} and is denoted by
$i(l)j$, if it belongs to the hyperedge $P_{ij}$ and has distance $l$
from the vertex $i$, where $l\in\{1,2,\ldots,k-1\}$. Note that
$i(l)j=j(k-l)i$, $i(0)j\isdef i$ and $i(k)j\isdef j$. Also any vertex $i=i(0)j$ of $G^\frac{1}{k}$ is called \emph{terminal vertex} or brifely \emph{t-vertex}.\\
Note that for $k=1$, we have $G^\frac{1}{1}=G^1=G$, and if the graph
$G$ has $v$ vertices and $e$ edges, then the graph $G^\frac{1}{k}$
has $v+(k-1)e$ vertices and
$ke$ edges.\\
Now we can define the fractional power of a graph as follows.
\begin{defin}
{Let $G$ be a graph and $m,n\in \mathbb{N}$. The graph
$G^{\frac{m}{n}}$ is defined by the $m-$power of the $n-$subdivision
of $G$. In other words
$G^{\frac{m}{n}}\isdef (G^{\frac{1}{n}})^m$. }
\end{defin}
Note that the graphs $(G^{\frac{1}{n}})^m$ and
$(G^{m})^{\frac{1}{n}}$ are different graphs
and in this paper, we only consider the graphs $(G^{\frac{1}{n}})^m$ for all positive integers $m$ and $n$. Furthermore, we assume that $V(G^{\frac{m}{n}})=V(G^{\frac{1}{n}})$.\\
As you know, for any graph $G$, we have
$\chi(G)\leq\chi(G^2)\leq\chi(G^3)\leq\ldots\leq\chi(G^d)=|V(G)|$,
where $d$ is the diameter of $G$. In other words, the
chromatic number of a graph $G$ increases when we replace $G$ by
its powers. The problem of coloring of graph powers, specially planar
graph powers, is studied in the literature (see for example [1, 2, $5-11$]).\\
On the other hand, it is easy to verify that by replacing any graph by
its subdivisions, the chromatic number of graph decreases. In other
words, $\chi(G)\geq\chi(G^\frac{1}{k})$ for any $k\in \mathbb{N}$. Now the following
question arises naturally.\\
\textbf{Question.} \emph{What happened for the chromatic number,
when we consider the fractional power of a graph, specially less
than (or more than) one power?}\\
In this paper we answer to this
question for the less than one power of a graph and find the
chromatic number of $G^{\frac{m}{n}}$ for
rational number $\frac{m}{n}<1$.\\
Another motivation for us to define the fractional power of a graph is the \emph{Total Coloring Conjecture}. Vizing (1964) and, independently, Behzad (1965) conjectured that the total chromatic number of a simple graph G never exceeds $\Delta+2$ (and thus $\Delta+1\leq\chi''(G)\leq\Delta+2$) [3,12]. By virtue of Definition 1, the total chromatic number of a simple graph $G$ is equal to $\chi(G^{\frac{2}{2}})$ and $\omega(G^{\frac{2}{2}})=\Delta+1$ when $\Delta\geq2$. Thus, the Total Coloring Conjecture can be reformulated as follows.\vspace{.2cm}\\
\emph{\textbf{Total Coloring Conjecture.}} \emph{Let $G$ be a simple graph and $\Delta(G)\geq2$. Then $\chi(G^{\frac{2}{2}})\leq\omega(G^{\frac{2}{2}})+1$.}\vspace{.2cm}\\
In the next section, at first, we calculate the clique
number of $G^{\frac{m}{n}}$ when $\frac{m}{n}<1$. Next, the chromatic number of $G^{\frac{m}{n}}$ is calculated when $\Delta(G)=2$. Later, we show that $\chi(G^{\frac{2}{n}})=\omega(G^{\frac{2}{n}})$ for any positive integer
$n\geq 3$ when $\Delta(G)\geq3$. Finally we show that $\chi(G^{\frac{m}{m+1}})=\omega(G^{\frac{m}{m+1}})$ when  $G$ is a connected graph with $\Delta(G)\geq3$ and $m\in \mathbb{N}$, which leads us to claim the following conjecture.\\
\emph{\noindent\textbf{Conjecture A.} Let $G$ be a connected graph with $\Delta(G)\geq3$, $n,m\in \mathbb{N}$ and $1<m<n$. Then
\[\chi(G^{\frac{m}{n}})=\omega(G^\frac{m}{n}).\]}
\section{Main Results}
Here, we find the clique number of $G^{\frac{m}{n}}$ when
$\frac{m}{n}<1$.
\begin{thm}
{Let $G$ be a graph, $n,m\in \mathbb{N}$ and $m<n$. Then
\[\omega(G^{\frac{m}{n}})=\left\{\begin{array}{ll}m+1&\Delta(G)=1\\\frac{m}{2}
\Delta(G)+1&\Delta(G)\geq2,m\equiv
0\medspace(mod\thinspace2)\\\frac{m-1}{2}\Delta(G)+2&\Delta(G)\geq2,m\equiv
1\medspace(mod\thinspace2).\end{array}\right.\] }
\end{thm}
\begin{proof}
{We consider three cases:\\
\emph{Case 1.} $\Delta(G)=1$\\
In this case, each connected component of $G$ is $K_2$ or $K_1$. Therefore,
$\omega(G^{\frac{m}{n}})=\omega(K_2^{\frac{m}{n}})=\omega(P_{n+1}^m)$.
It is clear that $\omega(P_{n+1}^m)=m+1$.\\
\emph{Case 2.} $\Delta(G)\geq2$, $m\equiv 0\medspace(mod\thinspace2)$\\
Choose a t-vertex $x$ of $G^{\frac{m}{n}}$ such that
$d_G(x)=\Delta(G)$. Consider the set
$S=\{y:d_{G^{\frac{m}{n}}}(y,x)\leq \frac{m}{2}\}$. Obviously, this
set is a clique with $\frac{m}{2} \Delta(G)+1$ vertices. Now, we show that
any clique of $G^{\frac{m}{n}}$ has at most $\frac{m}{2}
\Delta(G)+1$ vertices. Let $C$ be a maximal clique of $G^{\frac{m}{n}}$. The distance of any two t-vertices in $G^{\frac{m}{n}}$ is at least two. Therefore, $C$ has at most one t-vertex. First, suppose that $C$ has not any t-vertex.
Thus, all vertices of $C$ belong to a hyperedge and hence $|C|\leq m+1\leq
\frac{m}{2} \Delta(G)+1$. Now, suppose that $C$ has a t-vertex $x$. Let $y$ has the greatest distance from $x$ among the other vertices of $C$. If $d_{G^{\frac{m}{n}}}(y,x)=d>\frac{m}{2}$, then
$C$ has at most $m-d$ vertices in common with any hyperedge that is
adjacent with $x$ and doesn't contain $y$. Therefore, \[|C|=(d(x)-1)(m-d)+d+1\leq(\Delta(G)-2)(m-d)+m+1\]
\[\leq(\Delta(G)-2)\frac{m}{2}+m+1=\Delta(G)\frac{m}{2}+1.\]
In other case, if $d\leq \frac{m}{2}$, then $C$ has at most $\frac{m}{2}$
vertices in common with any hyperedge that is
adjacent to $x$. So $|C|\leq \frac{m}{2}\Delta(G)+1$.\\
\emph{Case 3.} $\Delta(G)\geq2$, $m\equiv 1\medspace(mod\thinspace2)$\\
Choose a t-vertex $x$ of $G^{\frac{m}{n}}$ such that
$d_G(x)=\Delta(G)$. Consider the sets
$S_1=\{y:d_{G^{\frac{m}{n}}}(y,x)\leq \frac{m-1}{2}\}$ and $S_2=\{z:d_{G^{\frac{m}{n}}}(z,x)=\frac{m+1}{2}\}$. Obviously each
element of $S_2$ along with the vertices of $S_1$, form a clique of size
$\frac{m-1}{2}\Delta(G)+2$. Now, similar to case 2, one can prove that
any clique of $G^{\frac{m}{n}}$ has at most $\frac{m-1}{2}
\Delta(G)+2$ vertices.}
\end{proof}
F. Kramer and H. Kramer gave in [9] the following characterization
of the graphs for which $\chi(G^m)=m+1$.
\begin{alphthm}
{\emph{[9]} Let $G=(V,E)$ be a simple and connected graph and $m\in \mathbb{N}$. We have $\chi(G^m)=m+1$ if and only if the
graph $G$ is satisfying one of the following conditions:\\
$(a)$ $|V|=m+1$,\\
$(b)$ $G$ is a path of length greater than m,\\
$(c)$ $G$ is a cycle of length a multiple of $m+1$.}
\end{alphthm}
To find the chromatic number of $G^{\frac{m}{n}}$ in the case
$\frac{m}{n}<1$, we only consider the connected graphs in two cases
$\Delta(G)=2$ and $\Delta(G)>2$. Theorem A can be considered as a result of
Theorems 1 and 2.
\begin{thm}
{Let $m,k\in \mathbb{N}$ and $k\geq3$. Then\\
$(a)$ $\chi(C_k^{m})=\left\{\begin{array}{ll}k&m\geq\lfloor\frac{k}{2}\rfloor\\
    \lceil\frac{k}{\lfloor\frac{k}{m+1}\rfloor}\rceil&m<\lfloor\frac{k}{2}\rfloor,\end{array}\right.$\\
$(b)$  $\chi(P_k^{m})=min\{m+1,k\}$.
}
\end{thm}
\begin{proof}
{First, we prove part $(a)$. It is easy to see that $\chi(C_k^{m})=\chi(K_k)=k$ provided that $m\geq\lfloor\frac{k}{2}\rfloor$. Thus, suppose that $m<\lfloor\frac{k}{2}\rfloor$. Obviously, we can show that
$\omega(C_k^{m})=m+1$. In fact, any set of $m+1$ consecutive vertices
of $C_k^{m}$ is a clique of maximum size. Therefore, $\chi(C_k^{m})\geq
m+1$. Let $V(C_k)=\{1,2,\ldots,k\}$ and $E(C_k)=\{\{1,2\},\{2,3\},\ldots,\{k-1,k\},\{k,1\}\}$. If $m+1|k$, then
$\lceil\frac{k}{\lfloor\frac{k}{m+1}\rfloor}\rceil=m+1$. Hence, to have a
proper coloring of $C_k^{m}$ with $m+1$ colors, it's enough to color
the vertices $i+j(m+1)$ with the color $i$ when $1\leq i\leq m+1$
and $0\leq j\leq \frac{k}{m+1}-1$. Finally, suppose that $m+1\nmid
k$. Because any set of $m+1$ consecutive vertices of $C_k^{m}$ is a
clique of maximum size, then any stable set and specially any color
class of a proper coloring, has at most
$\lfloor\frac{k}{m+1}\rfloor$ vertices. Therefore,
$\chi(C_k^{m})\lfloor\frac{k}{m+1}\rfloor\geq k$ that concludes
$\chi(C_k^{m})\geq
\lceil\frac{k}{\lfloor\frac{k}{m+1}\rfloor}\rceil$. Now we show that
$\lceil\frac{k}{\lfloor\frac{k}{m+1}\rfloor}\rceil$ colors are
enough. Let $k=(m+1)q+r$ and $1\leq r<m+1$. Then we have
$\lceil\frac{k}{\lfloor\frac{k}{m+1}\rfloor}\rceil=m+1+\lceil\frac{r}{q}\rceil$.
Suppose that $r=qq_{1}+r_1$ and $0\leq r_1<q$. In what follows,
we partite $V(C_k^{m})$ to $q$ subsets, such that each of these subsets contains
$m+1+\lceil\frac{r}{q}\rceil$ or $m+1+\lfloor\frac{r}{q}\rfloor$
consecutive vertices and then we color the vertices
of each subset with the colors $\{1,2,\ldots,\chi\}$ consecutively.\\
Precisely, to have a proper coloring of $C_k^{m}$ with $m+1+\lceil\frac{r}{q}\rceil$ colors, we consider two cases.\\
\emph{\textbf{Case 1.}} If $1\leq r_1<q$, it's enough
to color the vertices $i+j(m+1+\lceil\frac{r}{q}\rceil)$ with the
color $i$ when $1\leq i\leq m+1+\lceil\frac{r}{q}\rceil$ and $0\leq
j\leq r_1-1$. Also, color the vertices
$i+j(m+1+\lfloor\frac{r}{q}\rfloor)+r_1(m+1+\lceil\frac{r}{q}\rceil)$
with color $i$ when $1\leq i\leq m+1+\lfloor\frac{r}{q}\rfloor$ and
$0\leq j\leq q-r_1-1$.\\
\emph{\textbf{Case 2.}} If $r_1=0$, then $\lfloor\frac{r}{q}\rfloor=\lceil\frac{r}{q}\rceil$ and it's enough
to color the vertices $i+j(m+1+\lceil\frac{r}{q}\rceil)$ with the
color $i$ when $1\leq i\leq m+1+\lceil\frac{r}{q}\rceil$ and $0\leq
j\leq q-1$.\\
Now we prove part $(b)$. We know that $d_{P_k}(x,y)\leq k-1$. So if $m+1\geq k$, then
$P_k^{m}=K_k$ and $\chi(C_k^{m})=k$. Now suppose that $m+1<k$. We show that $m+1$ colors are enough. Let $l=s(m+1)>k+m+1$ and
consider the graph $C_l^m$. In view of the first part, this graph has
a proper coloring with $m+1$ colors. Because any $k$ consecutive
vertices of $C_l^m$ induce a subgraph isomorphic to $P_k^{m}$, so any
proper coloring of $C_l^m$ gives us a proper coloring of $P_k^{m}$.}
\end{proof}
The following corollary is a direct consequence of Theorem 2.
\begin{cor}
{Let $m,n,k\in \mathbb{N}$ and $k\geq3$. Then\\
$(a)$ $\chi(C_k^{\frac{m}{n}})=\left\{\begin{array}{ll}nk&m\geq\frac{nk}{2}\\
    \lceil\frac{nk}{\lfloor\frac{nk}{m+1}\rfloor}\rceil&m<\frac{nk}{2},\end{array}\right.$\\
$(b)$ $\chi(P_k^{\frac{m}{n}})=min\{m+1,(k-1)n+1\}$.
}
\end{cor}
\begin{proof}
{ Note that $C_k^{\frac{m}{n}}=C_{kn}^m$ and
$P_k^{\frac{m}{n}}=P_{(k-1)n+1}^m$. Hence, in view of Theorem 2, we can easily
conclude this corollary.}
\end{proof}
In Theorem 2 and Corollary 1 we consider the graphs with maximum
degree 2. Hereafter, we focus on the graphs with maximum degree greater than 2.
\begin{rem}
{Let $G$ be a connected graph and $n$ be a positive integer greater than 1. Then easily one can see that at most three colors are enough to achieve a proper coloring of $G^{\frac{1}{n}}$. Precisely,
\[\chi(G^{\frac{1}{n}})=\left\{\begin{array}{cl}3&n\equiv 1(mod\thinspace2)\ \& \ \chi(G)\geq3\\2&otherwise.\end{array}\right.\]
}
\end{rem}
In Lemmas $1-4$, we construct the steps of the proof of
Theorem 3 which states that
$\chi(G^{\frac{2}{n}})=\omega(G^{\frac{2}{n}})$ for any positive integer $n$ greater than 2.
\begin{lem}
{Let $G$ be a graph, $n,m\in \mathbb{N}$ and $m<n$. If
$\chi(G^{\frac{m}{n}})=\omega(G^{\frac{m}{n}})$, then
$\chi(G^{\frac{m}{n+m+1}})=\omega(G^{\frac{m}{n+m+1}})$.  }
\end{lem}
\begin{proof}
{Note that $\omega(G^{\frac{m}{n}})=\omega(G^{\frac{m}{n+m+1}})$. Hence, it is sufficient
to show that $\chi(G^{\frac{m}{n+m+1}})\leq\chi(G^{\frac{m}{n}})$. Let
$f:V(G^{\frac{m}{n}})\rightarrow [\chi]$ be a proper coloring of
$G^{\frac{m}{n}}$ with $\chi=\chi(G^{\frac{m}{n}})$ colors. Choose the edge
incident to vertex $i$ in the hyperedge $P_{ij}$, for any $i<j$. Then replace each of these edges by an $(m+2)-$path, to
make $G^{\frac{m}{n+m+1}}$. Now we color the new vertices on the
hyperedge $P_{ij}$, by the set of colors $S_{ij}=\{f(v)|v\in P_{ij}\}$,
such that the resulting
coloring ($f':V(G^{\frac{m}{n+m+1}})\rightarrow [\chi]$) is a proper
coloring of $G^{\frac{m}{n+m+1}}$. Precisely, suppose that an edge
between $i$ and $i(1)j$ is replaced by an $(m+2)-$path
$(i,i_1,i_2,\ldots,i_{m+1},i(1)j)$. Now use the color of $i(k)j$ to
color the new vertex $i_k$ when $1\leq k\leq m+1$. Easily we
can show that this coloring is a proper coloring. Consider
$f'(N_m(i_k)\setminus \{i_k\})$ that is the colors of all vertices with distance
$d\in\{1,2,\ldots,m\}$ from $i_k$. Because $f'(i_l)=f(i(l)j)$ so
$f'(N_m(i_k)\setminus \{i_k\})\subseteq f(N_m(i(k)j)\setminus \{i(k)j\})$. Then the color of $i_k$ is different from the color of any other vertex in $N_m(i_k)$. In addition, for any two vertices $x$ and
$y$ of $G^{\frac{m}{n}}$ with $f(x)=f(y)$, we have
\[d_{G^{\frac{m}{n+m+1}}}(x,y)\geq d_{G^{\frac{m}{n}}}(x,y)>m.\]
Hence, the coloring $f'$ is a proper coloring for $G^{\frac{m}{n+m+1}}$.
}
\end{proof}
\begin{lem}
{Let $G$ be a connected graph of order $n$. Then there is
a labeling $f:V(G)\rightarrow[n]$, such that,\\
$(a)$ for any two distinct vertices $i$ and $j$ of $G$, $f(i)\neq f(j)$ and\\
$(b)$ for any vertex $x$, if $f(x)\neq n$, then $f(x)<f(y)$ for some vertex $y\in N(x)$.}
\end{lem}
\begin{proof}
{A labeling $f$ of $V(G)$ is called a \emph{coloring order} of $G$, if it satisfies two properties $(a)$ and $(b)$. The proof is by induction on $n$. Obviously, the lemma is correct for $n=1$. Assume that there is a
coloring order for any connected graph of order $n-1$. Because $G$
is connected, so there is a vertex $x$, such that $G[V(G)\setminus \{x\}]$ is
connected. Then $G[V(G)\setminus \{x\}]$ has a coloring order like
$f:G[V(G)\setminus \{x\}]\rightarrow[n-1]$. Now we can define a coloring order
for $G$ as follows:\\
\[f'(i)=\left\{\begin{array}{ll}1&i=x\\f(i)+1&i\neq x.\end{array}\right.\]
Clearly one can check that $f'$ is a coloring order for $G$. }
\end{proof}
\begin{lem}
{Let $G$ be a connected graph with $\Delta(G)\geq3$. Then
$\chi(G^{\frac{2}{3}})=\Delta(G)+1=\omega(G^{\frac{2}{3}})$.}
\end{lem}
\begin{proof}
{It is well-known that any connected graph with maximum degree $\Delta$ is a
subgraph of a connected $\Delta-$regular graph. Hence, it is sufficient to prove the lemma
for the connected $\Delta-$regular graph $G$ of order $n$.\\
Color all t-vertices by the color 0. We show that the colors of
$C=\{1,2,\ldots,\Delta\}$ are enough to color i-vertices.
Suppose that $f$ is a coloring order of $G$ and $f^{-1}(i)$ is
denoted by $f_i$ . We color the i-vertices of $N(f_1)$, $N(f_2)$,
$\ldots$ and $N(f_n)$ consecutively in $n$ steps:\\
\emph{Step 1.} We color optionally the i-vertices of $N(f_1)$
with the colors of $C$.\\
\emph{Step 2.} If two vertices $f_1$ and $f_2$ are not adjacent in
$G$, then we can color optionally the i-vertices of $N(f_2)$. If
$f_1$ and $f_2$ are adjacent in $G$, then at first we color the
i-vertex $f_2(1)f_1$, with color $c_1$ which is different from the
color of $f_1(1)f_2$. Then we color the other i-vertices of $N(f_2)$
with the colors of $C\setminus\{c_1\}$ optionally.\\
\emph{Step $k$ $(3\leq k\leq n-1)$.}
By considering the colored vertices in previous steps, there are $\Delta$ or $\Delta-1$ colors to color each vertex of $N(f_k)$. Now by using Hall's Theorem [4, page 419], we can find a perfect matching between the vertices of $N(f_k)$ and the colors $C=\{1,2,\ldots,\Delta\}$ which leads us to a coloring of the vertices of $N(f_k)$.\\
\emph{Step $n$.} Let
$N_G(f_n)=\{f_{i_1},f_{i_2},\ldots,f_{i_{\Delta}}\}$. Suppose that
$c_j$ is the color of $f_{i_j}(1)f_n$, for any $1\leq j\leq \Delta$. If
$|\{c_j|1\leq j\leq \Delta\}|=k\geq2$, then we select $k$ vertices of
\[\{f_{i_1}(1)f_n,f_{i_2}(1)f_n,\ldots,f_{i_{\Delta}}(1)f_n\}\]
with different colors and color their neighbors in $N(f_n)$, with a
disarrangement of that $k$ colors. After that, we can color the
remaining i-vertices of $N(f_n)$, with the colors of $C\setminus\{c_1,c_2,\ldots,c_{\Delta}\}$.\\
Now let $k=1$. In other words, assume that $c_1=c_2=\ldots=c_{\Delta}=a$.\\
To color the i-vertices of $N(f_n)$ in this case, we need to
change the color of one i-vertex of
$F=\{f_{i_1}(1)f_n,f_{i_2}(1)f_n,\ldots,f_{i_{\Delta}}(1)f_n\}$.\\
Consider the subgraph of $G^{\frac{2}{3}}$, induced by the vertices of colors $a$ and $b$ and let $H$ be the connected component of this subgraph, containing the vertex $f_{i_1}(1)f_n$. Clearly, we have $1\leq d_H(x)\leq 2$ for each $x\in V(H)$ and $d_H(f_{i_1}(1)f_n)=1$. Therefore, $H$ is a path and we can interchange $k$ to 2, by changing the colors $a$ and $b$ in $H$ and then we can color the
i-vertices of $N(f_n)$. }
\end{proof}
\begin{lem}
{Let $G$ be a connected graph with $\Delta(G)\geq3$. Then\\
$(a)$ $\chi(G^{\frac{2}{4}})=\Delta(G)+1=\omega(G^{\frac{2}{4}})$,\\
$(b)$ $\chi(G^{\frac{2}{5}})=\Delta(G)+1=\omega(G^{\frac{2}{5}})$.}
\end{lem}
\begin{proof}
{First, we prove part $(a)$. We use the proper coloring of $G^{\frac{2}{3}}$ which is defined in Lemma 3. Suppose that $f:V(G^{\frac{2}{3}})\rightarrow C$ is a
proper coloring of $G^{\frac{2}{3}}$ where
$C=\{0,1,2,\ldots,\Delta(G)\}$. Three colors are appeared  on each
hyperedge of $G^{\frac{2}{3}}$. Because $\Delta(G)+1\geq 4$, so for
each hyperedge $P_{ij}$, there is at least one color, denoted by
$c_{ij}$, which is not used for the coloring of the vertices of $P_{ij}$.
Now on each hyperedge $P_{ij}$, replace the edge between $i(1)j$ and $j(1)i$ by a $2-$path and color the new vertex by $c_{ij}$.
In other words, we have a proper coloring $f'$ for
$G^{\frac{2}{4}}$ as follows:\\
\[f'(x)=\left\{\begin{array}{cl}c_{ij}&x=i(2)j=j(2)i\\f(x)&otherwise.\end{array}\right.\]
Clearly, we can show that $f'$ is a proper coloring for
$G^{\frac{2}{4}}$.\\
Now, we prove part $(b)$. Let $f:E(G)\rightarrow C$ be a proper edge coloring of $G$, where $C=\{0,1,2,\ldots,\Delta(G)\}$. At first, we color the
i-vertices $i(1)j$ and $j(1)i$ with the same color of $f(ij)$. Then we
color all  t-vertices of $G^{\frac{2}{5}}$. Color the t-vertex
$i$, with one color of $C\setminus\{f(ij)|d_G(j,i)=1\}$ which is non-empty. Finally we color two remaining i-vertices of each
hyperedge. There are two cases for each hyperedge $P_{ij}$:\\
\emph{Case 1}. Two vertices $i$ and $j$ have the same color.\\
In this case, only two colors are used on the hyperedge $P_{ij}$.
Because $\Delta(G)+1\geq 4$, so there are at least two colors which
are not appeared on the vertices of $P_{ij}$, and we can color two
uncolored i-vertices of $P_{ij}$, with the remaining colors.\\
\emph{Case 2.} Two vertices $i$ and $j$ have different colors.\\
In this case, we color the i-vertex $i(2)j$ with the color of $j$
and the i-vertex $j(2)i$ with the color of $i$. Note that
$d_{G^\frac{2}{5}}(i(2)j,j)=2=d_{G^\frac{2}{5}}(j(2)i,i)$ and
$d_{G^\frac{2}{5}}(i(2)j,i)=1=d_{G^\frac{2}{5}}(j(2)i,j)$. Clearly,
one can show that this coloring is a proper coloring of
$G^{\frac{2}{5}}$.}
\end{proof}
Now we can prove the following theorem inductively by applying Theorem 1 and Lemmas 1, 3 and 4.
\begin{thm} {Let $G$ be a connected graph with $\Delta(G)\geq3$ and $n$ be a positive integer greater than 2. Then
$\chi(G^{\frac{2}{n}})=\Delta(G)+1=\omega(G^{\frac{2}{n}})$.  }
\end{thm}
Here, we show that the conjecture A is true for some rational number $\frac{m}{n}<1$.
\begin{thm}
{Let $G$ be a connected graph with $\Delta(G)\geq3$ and
$m\in \mathbb{N}$. Then $\chi(G^{\frac{m}{m+1}})=\omega(G^{\frac{m}{m+1}})$.
}
\end{thm}
\begin{proof}
{We prove this theorem by induction on $m$. Remark 1 and Lemma 3
show that the assertion holds for $m=1,2$. We prove that if the
assertion holds for $m=2k\in \mathbb{N}$, then it is correct for $m=2k+1$
and $m=2k+2$. Suppose that
$\chi(G^{\frac{2k}{2k+1}})=\omega(G^{\frac{2k}{2k+1}})=k\Delta(G)+1$.
To prove $\chi(G^{\frac{2k+1}{2k+2}})=k\Delta(G)+2$, we follow the same lines as in the proof of Lemma 4. Assume that
$f:V(G^{\frac{2k}{2k+1}})\rightarrow C$ is a proper coloring of
$G^{\frac{2k}{2k+1}}$ where $C=\{0,1,2,\ldots,k\Delta(G)\}$. Then we
can define a proper coloring $f':V(G^{\frac{2k+1}{2k+2}})\rightarrow
C\cup\{k\Delta(G)+1\}$ as follows:\\
\[f'(x)=\left\{\begin{array}{cl}k\Delta(G)+1&x=i(k+1)j\\f(x)&x=i(l)j,\thinspace l\leq k.\end{array}\right.\]
In fact, one i-vertex with the color $k\Delta(G)+1$, is added to
each hyperedge $P_{ij}$ of  $G^{\frac{2k}{2k+1}}$ between two
central i-vertices of $P_{ij}$ to construct
$G^{\frac{2k+1}{2k+2}}$. We show that $f'$ is a proper
coloring for $G^{\frac{2k+1}{2k+2}}$.\\
Suppose that $x$ and $y$ are two adjacent vertices of
$G^{\frac{2k+1}{2k+2}}$. Therefore, $d_{G^{\frac{1}{2k+2}}}(x,y)\leq
2k+1$. It is easy to check that $f'(x)\neq f'(y)$ when $f'(x)=k\Delta(G)+1$ or $f'(y)=k\Delta(G)+1$. Hence, we assume that $f'(x)\neq k\Delta(G)+1\neq f'(y)$. If $d_{G^{\frac{1}{2k+2}}}(x,y)\leq 2k$, then before
adding new i-vertices, we must have
$d_{G^{\frac{1}{2k+1}}}(x,y)\leq 2k$. Then $x$ and $y$ are
adjacent in $G^{\frac{2k}{2k+1}}$ and their colors are different
in $f$. Now suppose that $d_{G^{\frac{1}{2k+2}}}(x,y)=2k+1$.
Therefore, there is a path of length $2k+1$ between $x$ and $y$
which contains a new i-vertex with the color $k\Delta(G)+1$.
Hence, the distance of $x$ to $y$ in $G^{\frac{1}{2k+1}}$ is at
most $2k$. Thus, $x$ and $y$ are adjacent in
$G^{\frac{2k}{2k+1}}$ and their colors are different.
Consequently, $f'$ is a proper coloring for $G^{\frac{2k+1}{2k+2}}$.\\
To prove $\chi(G^{\frac{2k+2}{2k+3}})=(k+1)\Delta(G)+1$, suppose
that $f:V(G^{\frac{2k}{2k+1}})\rightarrow C$ is a proper coloring
of $G^{\frac{2k}{2k+1}}$ such that
$C=\{0,1,2,\ldots,k\Delta(G)\}$ and all t-vertices, have the same
color 0. To construct $G^{\frac{2k+2}{2k+3}}$ from $G^{\frac{2k}{2k+1}}$, subdivide the central edge $\{i(k)j,j(k)i\}$ of any hyperedge $P_{ij}$ to three edges $\{i(k)j,i[k+1]j\}$, $\{i[k+1]j,i[k+2])j\}$ and $\{i[k+2]j,j(k)i\}$ which contain two new i-vertices $i[k+1]j$ and $i[k+2]j$. Now let $S$ be the set of
the new i-vertices and all t-vertices. Clearly we can show that
$G^{\frac{2k+2}{2k+3}}[S]$ is isomorphic to $G^{\frac{2}{3}}$.
Thus, in view of Lemma 3, we have
$\chi(G^{\frac{2k+2}{2k+3}}[S])=\Delta(G)+1$. Now let
$f':V(G^{\frac{2k+2}{2k+3}}[S])\rightarrow\{0,c_1,c_2,\ldots,c_{\Delta}\}$
be a proper coloring of $G^{\frac{2k+2}{2k+3}}[S]$ such that the
color of all t-vertices is 0 and
$\{0,c_1,c_2,\ldots,c_{\Delta}\}\cap C=\{0\}$. Now we can define
a proper coloring $f''$ for $G^{\frac{2k+2}{2k+3}}$, as
follows:\\
\[f''(x)=\left\{\begin{array}{cl}f'(x)&x\in S\\f(x)&x=i(l)j,\thinspace l\leq k.\end{array}\right.\]
Suppose that $x$ and $y$ are two adjacent vertices of
$G^{\frac{2k+2}{2k+3}}$. If $x,y\in S$, then they are adjacent in
$G^{\frac{2k+2}{2k+3}}[S]$. Hence, $f''(x)=f'(x)\neq
f'(y)=f''(y)$. If only one of them is in $S$, obviously we have
$f''(x)=f'(x)\neq f''(y)$ or $f''(y)=f'(y)\neq f''(x)$. Now
suppose that $x,y\in V(G^{\frac{2k+2}{2k+3}})\setminus S$.
Therefore, $d_{G^{\frac{1}{2k+3}}}(x,y)\leq 2k+2$. If
$d_{G^{\frac{1}{2k+3}}}(x,y)\leq 2k$, then before adding central
two i-vertices to each hyperedge, we have
$d_{G^{\frac{1}{2k+1}}}(x,y)\leq 2k$. Then $x$ and $y$ are
adjacent in $G^{\frac{2k}{2k+1}}$ and their colors are different.
Now suppose that $d_{G^{\frac{1}{2k+3}}}(x,y)\geq 2k+1$.
Therefore, there is a path of length $2k+1$ or $2k+2 $ between
$x$ and $y$, which contains two new i-vertices of $S$. Hence, the
distance of $x$ to $y$ in $G^{\frac{1}{2k+1}}$ is at most $2k$.
Thus, $x$ and $y$ are adjacent in $G^{\frac{2k}{2k+1}}$ and their
colors are different.}
\end{proof}
\begin{cor}
{Let $G$ be a connected graph with $\Delta(G)\geq3$ and $k,m\in \mathbb{N}$. Then
$\chi(G^{\frac{m}{k(m+1)}})=\omega(G^{\frac{m}{k(m+1)}})$.  }
\end{cor}
\begin{proof}
{We can result this corollary by using Lemma 1 and Theorem 4. }
\end{proof}
\begin{thm}
{Let $G$ be a connected graph and $1<m\in \mathbb{N}$.\\
$(a)$ If $m$ is an even integer and $\Delta(G)\geq4$, then for any
integer $n\geq 2m+2$
\[\chi(G^{\frac{m}{n}})=\omega(G^\frac{m}{n}).\]
$(b)$ If $m$ is an odd integer and $\Delta(G)\geq5$, then for any
integer $n\geq 2m+2$
\[\chi(G^{\frac{m}{n}})=\omega(G^\frac{m}{n}).\]
}
\end{thm}
\begin{proof}
{The proofs of parts $(a)$ and $(b)$ are similar. Thus, we only prove the first part.\\
The proof is by induction on $n$. Corollary 2 shows that $(a)$ is
holds for $n=2m+2$. Suppose that
$\chi(G^{\frac{m}{n}})=\omega(G^\frac{m}{n})$ for some $n\geq
2m+2$. Let $f:V(G^{\frac{m}{n}})\rightarrow [\chi]$ be a proper
coloring of $G^{\frac{m}{n}}$ when $\chi=\chi(G^{\frac{m}{n}})$.
We extend this coloring to a proper coloring of
$G^{\frac{m}{n+1}}$. On each hyperedge $P_{ij}$, Consider $2m$
consecutive i-vertices $i(1)j,i(2)j,\ldots,i(2m)j$ and let $V_{ij}=\{i(1)j,i(2)j,\ldots,i(2m)j\}$. So $|f(V_{ij})|\leq 2m$.\\
We have
$\chi(G^{\frac{m}{n}})=\omega(G^\frac{m}{n})=\frac{m}{2}\Delta(G)+1\geq
2m+1$. Thus, for any hyperedge $P_{ij}$, there is a color which does
not assign to the vertices of $V_{ij}$, denoted by $c_{ij}$.
Subdivide the edge $\{i(m)j,i(m+1)j\}$ of each hyperedge
$P_{ij}$, to two edges and color the new i-vertex by $c_{ij}$. It
is easy to check that we have a proper coloring of
$G^{\frac{m}{n+1}}$.}
\end{proof}
We can divide Conjecture A to the following conjectures.\\
\emph{\noindent\textbf{Conjecture $A(m)$.} Let $G$ be a connected graph with $\Delta(G)\geq3$ and $m$ be a positive integer greater than 1. Then for any positive integer $n\geq m$, we have $\chi(G^{\frac{m}{n}})=\omega(G^\frac{m}{n})$.}\\
In view of Theorem 3, Conjecture $A(2)$ holds. In addition, by applying Lemma 1, we can show that $A(m)$ holds if it is correct only for any $n\in \{m+1,m+2,\ldots,2m+1\}$. The Conjectures $A$ and $A(m)$ remain open for any $m\geq3$.\\
\noindent {\bf Acknowledgment}\\
We would like to thank Professor Hossein Hajiabolhassan, M. Alishahi,  A. Taherkhani and S. Shaebani for their useful comments.


\end{document}